\documentclass[4pt]{article} 

\usepackage[utf8]{inputenc} 
\usepackage{geometry} 
\usepackage{graphicx} 

\usepackage[colorlinks,citecolor=red]{hyperref}
\usepackage{amsmath}
\usepackage{amsthm}
\usepackage{amsfonts}
\usepackage{dsfont}
\usepackage{mathrsfs}
\usepackage{latexsym}
\usepackage{graphicx}
\usepackage{amssymb}
\usepackage{float}
\usepackage{epsfig}
\usepackage{epstopdf}
\usepackage{color,xcolor}
\usepackage{booktabs} 
\usepackage{array} 
\usepackage{paralist} 
\usepackage{verbatim} 
\usepackage{subfig} 

\newtheorem{theorem}{Theorem}[section]
\newtheorem{lemma}[theorem]{Lemma}

\newtheorem{proposition}[theorem]{Proposition}

\newtheorem{conjecture}[theorem]{Conjecture}

\usepackage{fancyhdr} 
\usepackage{sectsty}
\allsectionsfont{\sffamily\mdseries\upshape} 

\usepackage[nottoc,notlof,notlot]{tocbibind} 
\usepackage[titles,subfigure]{tocloft} 

\title{{\bf Laplacian eigenvalue distribution, diameter and domination number of trees}}
\author{Jiaxin Guo, Jie Xue\thanks{Email:~guo\_jiaxin1@126.com(J. Guo),~xuejie@zzu.edu.cn(J. Xue),~rfliu@zzu.edu.cn(R. Liu).}, ~Ruifang Liu
\\
{\small School of Mathematics and Statistics, Zhengzhou University, 450001 Zhengzhou, China}
}
\date{} 

\begin{document}
\maketitle

\begin{abstract}
   For a graph $G$ with domination number $\gamma$, Hedetniemi, Jacobs and Trevisan [European Journal of Combinatorics 53 (2016) 66--71] proved that $m_{G}[0,1)\leq \gamma$,
   where $m_{G}[0,1)$ means the number of Laplacian eigenvalues of $G$ in the interval $[0,1)$.
  Let $T$ be a tree with diameter $d$. In this paper, we show that $m_{T}[0,1)\geq (d+1)/3$.
  However, such a lower bound is false for general graphs. 
  All trees achieving the lower bound are completely characterized. 
  Moreover, for a tree $T$, we establish a relation between the Laplacian eigenvalues, the diameter and the domination number 
  by showing that the domination number of $T$ is equal to $(d+1)/3$ if and only if it has exactly $(d+1)/3$ Laplacian eigenvalues less than one.
  As an application, it also provides a new type of trees, which show the sharpness of an inequality due to Hedetniemi, Jacobs and Trevisan.
\end{abstract}

\section{Introduction}

Given a graph $G$, the Laplacian matrix of $G$ is $L(G)=D(G)-A(G)$, where $D(G)$ and $A(G)$ are the diagonal degree matrix and adjacency matrix
of $G$. The eigenvalues of $L(G)$ are called the Laplacian eigenvalues of $G$. 
Clearly, $L(G)$ is positive semi-definite, and 0 is one of its eigenvalues with eigenvector $\mathtt{1}$.
Usually, the Laplacian eigenvalues of $G$ are written as $0=\mu_{1}(G)\leq \mu_{2}(G)\leq \cdots\leq \mu_{n-1}(G)\leq \mu_{n}(G)$.
It can be represented as a multiset
$${\it{SPEL}}(G)=\{\mu_{1}(G),\mu_{2}(G), \ldots,\mu_{n-1}(G),\mu_{n}(G)\}.$$
As we know, the Laplacian eigenvalue of $G$ is at most $n$.
Hence, for any $1\leq i\leq n$, the Laplacian eigenvalue $\mu_{i}(G)$ belongs to the interval $[0,n]$.
For an interval $I\subseteq [0,n]$, we use $m_{G}I$ to denote the number of the Laplacian eigenvalues which belong to the interval $I$.
Then $m_{G}[0,1)$ means the number of its Laplacian eigenvalues less than 1, that is,
$$m_{G}[0,1)=\#\{\mu\in {\it{SPEL}}(G): 0\leq \mu<1\}.$$

The Laplacian matrix is a classic matrix which can be used to represent a graph. 
In 1973, Fiedler presented some relations between the Laplacian eigenvalues and the structure of graphs \cite{Fiedler1973}.
Since then, a number of papers were devoted to studying the Laplacian eigenvalues of graphs. 
The readers can find comprehensive overviews of theoretical and practical issues on the topic in the book \cite{Molitierno2012} by Molitierno.

The distribution of Laplacian eigenvalues can give us useful information about the structure of the graph. 
Grone, Merris and Sunder \cite{Grone1990} proved that $m_{G}[0,1)\geq q(G)$, where $q(G)$ is the number of quasi-pendant vertices in $G$.
If $G$ is a graph on $n$ vertices with $n>2q(G)$, Merris \cite{Merris1991} showed that $m_{G}(2,n]\geq q(G)$. 
Guo and Tan \cite{Guo2001} improved this bound by showing that if $n>2\nu(G)$, then $m_{G}(2,n]\geq \nu(G)$, 
where $\nu(G)$ is the matching number of $G$. If the independence number of $G$ is $\alpha(G)$, 
Ahanjideh, Akbari, Fakharan and Trevisan \cite{Ahanjideh2022} obtained that $m_{G}[0,1)\leq \alpha(G)\leq m_{G}[0,n-\alpha(G)]$.
Wang, Yan, Fang, Geng and Tian \cite{Wang2020} presented that $m_{G}(n-1,n]\leq \kappa(G)$ and $m_{G}(n-1,n]\leq \chi(G)-1$, where $\kappa(G)$ and $\chi(G)$
are respectively the vertex-connectivity and the chromatic number of $G$.

A dominating set in a graph $G$ is a vertex set $S\subseteq V(G)$ such that any vertex in $V(G)\backslash S$ is adjacent to a vertex of $S$.
The domination number $\gamma(G)$ is the minimum cardinality of a dominating set of $G$. 
For a tree $T$ on $n$ vertices, Zhou, Zhou and Du \cite{Zhou2015} showed that $m_{T}[0,2)\leq n-\gamma(T)$.
In fact, this result can be extended to general graphs. 
Cardoso, Jacobs and Trevisan \cite{Cardoso2017} presented that $m_{G}[0,2)\leq n-\gamma(G)$ if $G$ is an isolate-free graph.
Recently, Hedetniemi, Jacobs and Trevisan \cite{Hedetniemi2016} established the relation between $m_{G}[0,1)$ and $\gamma(G)$.

\begin{theorem}[\cite{Hedetniemi2016}]\label{the_1}
If $G$ is a graph with domination number $\gamma$, then $m_{G}[0,1)\leq \gamma$.
\end{theorem}

The above theorem gives an upper bound for the number of Laplacian eigenvalues less than 1. At the same time, it also provides a lower bound for domination number.
In \cite{Hedetniemi2016}, the authors gave a tree $T$ of order $65$ whose domination number $\gamma$ is strictly greater than $m_{T}[0,1)$.
On the other hand, it was also mentioned that the above bound is sharp. 
In order to provide further explanations for the sharpness, we will construct a new infinite family of trees, which satisfy the equality in Theorem \ref{the_1}.
To this end, we consider the relation between the Laplacian eigenvalues and the diameter of trees. 
The first result of the paper establishes a lower bound for the number of Laplacian eigenvalues of trees in the interval $[0,1)$.

\begin{theorem}\label{the_2}
   If $T$ is a tree with diameter $d$, then $m_{T}[0,1)\geq (d+1)/3$.
\end{theorem} 

According to Theorem \ref{the_2}, one can see that a tree will have many small Laplacian eigenvalues if its diameter is large.
For example, the diameter of a perfect binary tree of order $n$ is equal to $2(\log_{2}(n+1)-1)$, and it has at least $(2\log_{2}(n+1)-1)/3$ Laplacian eigenvalues less than one.

It is somewhat surprising that the lower bound in Theorem \ref{the_2} is also a lower bound for domination number of trees.
Numerous bounds for the domination number have been reported in the literature. 
One of the best-known lower bounds is that the domination number of a graph is at least $(d+1)/3$
(for this bound and other results on domination number, see the monograph \cite{Haynes1998} by Haynes, Hedetniemi and Slater).
Here we restate the bound for trees only.

\begin{theorem}[\cite{Haynes1998}]\label{the_3}
   If $T$ is a tree with diameter $d$ and domination number $\gamma$, then $\gamma\geq (d+1)/3$.
\end{theorem}

In \cite{Xue2020}, Xue, Liu, Yu and Shu established a spectral characterization when the domination number of a tree attains the lower bound. 
It was proved that a tree $T$ satisfying $\gamma=(d+1)/3$ if and only if it contains a Laplacian eigenvalue of multiplicity $n-d$.
In this paper, we give a new spectral characterization by using the number of Laplacian eigenvalues less than 1.
The experiments indicate that some trees achieve the lower bounds in Theorems \ref{the_2} and \ref{the_3} at the same time.
For example, the star $K_{1,n-1}$ has Laplacian eigenvalues $n, 1^{n-2}, 0$. It is obvious that the star $K_{1,n-1}$ satisfies 
$m_{K_{1,n-1}}[0,1)=\gamma=(d+1)/3=1$. In the general case, we obtain the following result.

\begin{theorem}\label{the_4}
   Let $T$ be a tree with diameter $d$ and domination number $\gamma$. Then $\gamma=(d+1)/3$ if and only if $m_T[0,1)=(d+1)/3$.
\end{theorem}

Let us denote $\Gamma(n,d)$ the set of all trees which have exactly $(d+1)/3$ Laplacian eigenvalues less than one. 
Such trees will be characterized in the following section. For any tree $T$ in $\Gamma(n,d)$, by Theorem \ref{the_4},
it is easy to see that $m_T[0,1)=(d+1)/3=\gamma$.
Then it also provides a type of trees satisfying the equality in Theorem \ref{the_1}.

The rest of the paper is organized as follows. In Section 2, we discuss the Laplacian eigenvalues less than 1, and give the proofs of Theorems \ref{the_2} and \ref{the_4}.
In Section 3, the lower bound in Theorem \ref{the_2} will be improved slightly.
Moreover, for a double starlike tree, the value of the number of Laplacian eigenvalues in $[0,1)$ is determined. 
In the light of experimental data and observations, some open problems are presented at the end of the paper.

\section{Laplacian eigenvalues less than 1}

Let us recall an interlacing result for symmetric matrices. Suppose that $A$ and $B$ are two symmetric matrices of order $n$.
Denote by $\lambda_{1}(A)\leq \lambda_{2}(A)\leq\cdots\leq \lambda_{n}(A)$ and $\lambda_{1}(B)\leq \lambda_{2}(B)\leq\cdots\leq \lambda_{n}(B)$
the eigenvalues of $A$ and $B$, respectively. If there is a vector $\mathtt{x}$ such that $A=B+\mathtt{x}\mathtt{x}^{t}$, then their eigenvalues satisfy 
\begin{eqnarray}\label{eq1}
   \lambda_{n}(B)\leq \lambda_{n}(A)\leq \lambda_{n-1}(B)\leq \cdots\leq\lambda_{2}(B)\leq \lambda_{2}(A)\leq \lambda_{1}(B)\leq\lambda_{1}(A). 
\end{eqnarray}
Let $G$ be a graph on $n$ vertices, and let $G'$ be a graph obtained from $G$ by adding an edge. 
Then the Laplacian matrices of $G$ and $G'$ satisfy the following equation:
\begin{eqnarray*}
   L(G')=L(G)+\begin{bmatrix}1\\-1\\0\\ \vdots\\0
   \end{bmatrix}\begin{bmatrix}1&-1&0& \cdots &0\end{bmatrix}.
\end{eqnarray*}
Using the interlacing result (\ref{eq1}), one can easily obtain the following property for the Laplacian eigenvalues of a graph when adding some edges.

\begin{proposition}[\cite{Cvetkovic1980}]\label{pro_1}
   Let $G$ be a graph on $n$ vertices, and let $G'$ be a graph obtained from $G$ by adding an edge. Then
   \begin{eqnarray}\label{eq2}
      0=\mu_1(G)=\mu_1(G')\leq \mu_2(G) \leq \mu_2(G')\leq\cdots\leq\mu_n(G)\leq \mu_n(G') .
   \end{eqnarray}
\end{proposition}

Proposition \ref{pro_1} states that the edge addition operation do not decrease Laplacian eigenvalues.
If $H$ is a spanning subgraph of a graph $G$, then it is easy to see that $m_{G}[0,1)\leq m_{H}[0,1)$.
Furthermore, using Proposition \ref{pro_1}, we establish the change for the number of Laplacian eigenvalues in $[0,1)$ when deleting pendant vertices.

\begin{lemma}\label{lem_1}
   Let $G$ be a graph on $n$ vertices, and let $G'$ be a graph obtained from $G$ by deleting a pendant vertex.
   If $m_{G'}[0,1)\geq k$, then $m_{G}[0,1)\geq k$.
\end{lemma}
\begin{proof}
Assume that $v$ is the deleted vertex, and $v$ is adjacent to the vertex $u\in V(G)$.
Then clearly $G'\cong G-v$. Consider the graph $G^{*}$ obtained from $G$ by deleting the pendant edge $uv$.
One can see that $G'$ is a subgraph of $G^{*}$, and $G^{*}$ is a subgraph of $G$.
Note also that $G^{*}$ can be treated as the union of $G'$ and an isolated vertex, that is, $G^{*}\cong G'\cup \{v\}$.
This implies that the Laplacian eigenvalues of $G^{*}$ consist of the Laplacian eigenvalues of $G'$ and an additional 0.
It follows that 
\begin{eqnarray}\label{eq3}
   \mu_{i}(G')=\mu_{i+1}(G^{*})
\end{eqnarray}
for $1\leq i\leq n-1$.
Since $m_{G'}[0,1)\geq k$, we have $\mu_{k}(G')<1$. According to (\ref{eq3}), it follows that 
\begin{eqnarray}\label{eq4}
   \mu_{k+1}(G^{*})=\mu_{k}(G')<1.
\end{eqnarray}
Since $G^{*}$ is obtained from $G$ by deleting an edge, by Proposition \ref{pro_1}, we obtain that
\begin{eqnarray}\label{eq5}
   \mu_{k+1}(G)\geq \mu_{k+1}(G^{*})\geq \mu_{k}(G).
\end{eqnarray}
Combining (\ref{eq4}) and (\ref{eq5}), we have $\mu_{k}(G)<1$. This implies that $m_{G}[0,1)\geq k$, as desired.
\end{proof}

As usual, $P_{n}$ denotes a path on $n$ vertices. Recall that the Laplacian eigenvalues of $P_{n}$ are: 
$$2-2\cos\frac{\pi i}{n}$$ where $i=0,1,\ldots,n-1$.
Now, we are ready to prove Theorem \ref{the_2}.

\begin{proof}[Proof of Theorem \ref{the_2}]
   If $T\cong P_{n}$, then its $k$-th smallest Laplacian eigenvalue is 
   $$\mu_{k}(T)=2-2\cos\frac{(k-1)\pi}{n}.$$
   Note that $\mu_{k}(T)<1$ if and only if $k-1<n/3$. This implies that $\mu_{k}(T)<1$ only if $k\leq \lceil n/3\rceil$.
   Then it follows that $m_{T}[0,1)=\lceil n/3\rceil\geq n/3=(d+1)/3$, as required. 

   Suppose now that $T\ncong P_{n}$, and $P_{d+1}$ is a diameter-path in the tree $T$. 
   Let $m$ be the number of vertices in $V(T)\backslash V(P)$. 
   It is easy to see that the path $P_{d+1}$ can be obtained from $T$ by successive deleting $m$ pendant vertices.
   This means that there are a series of trees
   $$T_{0}\cong T, T_{1},T_{2},\ldots,T_{m-1}, T_{m}\cong P_{d+1},$$
   such that, for $0\leq i\leq m-1$, $T_{i+1}$ is obtained from $T_{i}$ by deleting a pendant vertex.
   As mentioned above, we have $m_{P_{d+1}}[0,1)\geq (d+1)/3$. It follows from Lemma \ref{lem_1} that $m_{T_{m-1}}[0,1)\geq (d+1)/3$.
   By applying repeatedly Lemma \ref{lem_1}, we obtain that $m_{T}[0,1)\geq (d+1)/3$, which completes the proof.
\end{proof}

\begin{figure}[tbp]
   \centering
   \includegraphics[scale=0.47]{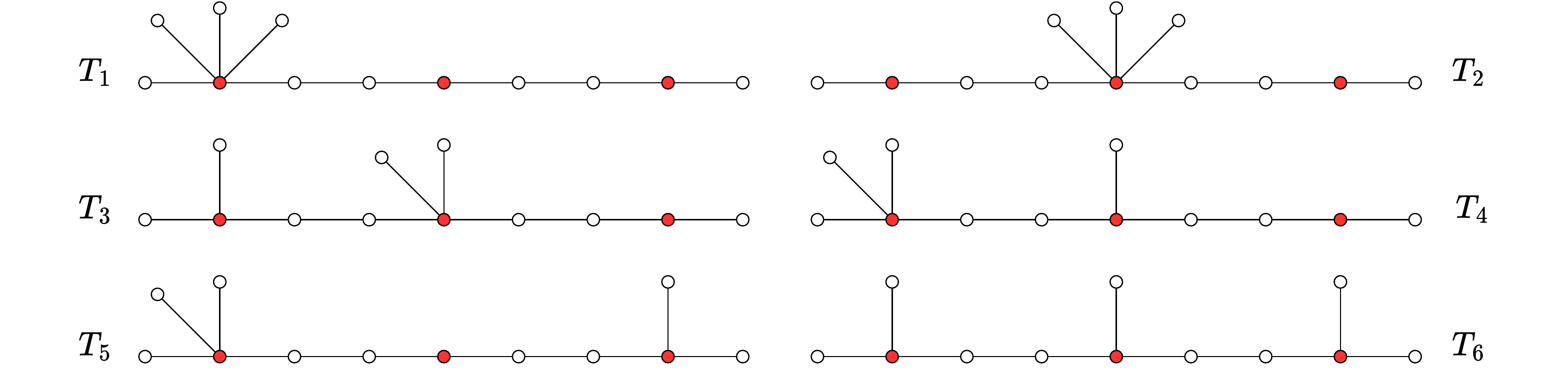}
   \caption{Trees in $\Gamma(12,8)$.}
   \label{fig_1}
   \end{figure}

In the following, we will characterize all trees which have exactly $(d+1)/3$ Laplacian eigenvalues less than one, that is, $m_{T}[0,1)=(d+1)/3$.

Let us define a set of trees with diameter $d\equiv 2~(\text{mod}~3)$. Suppose that $P_{d+1}=v_{1}v_{2}\cdots v_{d+1}$.
For $1\leq i\leq (d+1)/3$, we add $n_{i}\geq 0$ pendant vertices and join them to the vertex $v_{3i-1}$.
The resulting graph is obviously a tree, and denote it by $H_{d}(n_{1},n_{2},\ldots,n_{(d+1)/3})$.
Let 
$$\Gamma(n,d)=\left\{H_{d}(n_{1},n_{2},\ldots,n_{(d+1)/3}):\sum_{i=1}^{(d+1)/3}n_{i}=n-d-1~\text{and}~d\equiv 2~(\text{mod}~3)\right\}.$$
For any tree $T\in \Gamma(n,d)$, one can easy to check that $T$ satisfies the following properties:
\begin{itemize}
   \item[(P1)] its diameter is $d$;
   \item[(P2)] its domination number is $(d+1)/3$;
   \item[(P3)] it contains $n-d+1$ pendant vertices;
   \item[(P4)] the distance between any two pendant vertices is $2~(\text{mod}~3)$;
   \item[(P5)] 1 is a Laplacian eigenvalue of $T$ with multiplicity $n-d$ (see \cite[Theorem 2.12]{Xue2020}).
\end{itemize}
For example, in Figure \ref{fig_1}, we present all non-isomorphic trees in $\Gamma(12,8)$. 
The domination number of any tree in $\Gamma(12,8)$ is equal to 3, and the red vertices form a dominating set.
The Laplacian eigenvalues of these trees are established in Table \ref{tab_1}.
One can see that these trees have three Laplacian eigenvalues less than 1. 
If $d\equiv 2~(\text{mod}~3)$ and $n=d+1$, then $P_{d+1}$ is the only tree in $\Gamma(n,d)$.
Thus, we obtain that
$$\mu_{\frac{d+1}{3}}(P_{d+1})=2-2\cos\left(\frac{1}{3}-\frac{1}{d+1}\right)\pi<2-2\cos\frac{\pi}{3}=1,$$
and 
$$\mu_{\frac{d+1}{3}+1}(P_{d+1})=2-2\cos\frac{\pi}{3}=1,$$
which yields that $m_{P_{d+1}}[0,1)=(d+1)/3$.
In general, we show that any tree $T\in \Gamma(n,d)$ has exactly $(d+1)/3$ Laplacian eigenvalues less than one.

\begin{table}[tbp]   
     \begin{center}\caption{Laplacian eigenvalues of trees in $\Gamma(12,8)$.}\label{tab_1}
       \begin{tabular}{ccccccccccccc}
       \toprule
       & $\mu_{1}$ & $\mu_{2}$ & $\mu_{3}$ & $\mu_{4}$ & $\mu_{5}$ & $\mu_{6}$ & $\mu_{7}$ & $\mu_{8}$ & $\mu_{9}$ & $\mu_{10}$ & $\mu_{11}$ & $\mu_{12}$\\
       \midrule
       $T_{1}$& 6.055 & 3.814 & 3.301 & 2.572 & 1.760 & 1 & 1 & 1 & 1 & 0.414 & 0.084 & 0\\
       $T_{2}$& 6.107 & 3.532 & 3.438 & 2.347 & 2.195 & 1 & 1 & 1 & 1 & 0.260 & 0.121 & 0\\
       $T_{3}$& 5.187 & 4.172 & 3.464 & 2.600 & 2.200 & 1 & 1 & 1 & 1 & 0.274 & 0.102 & 0\\
       $T_{4}$& 5.103 & 4.335 & 3.420 & 2.641 & 2.094 & 1 & 1 & 1 & 1 & 0.316 & 0.091 & 0\\
       $T_{5}$& 5.098 & 4.233 & 3.582 & 2.773 & 1.847 & 1 & 1 & 1 & 1 & 0.388 & 0.078 & 0\\
       $T_{6}$& 4.461 & 4.199 & 4.000 & 2.714 & 2.239 & 1 & 1 & 1 & 1 & 0.300 & 0.088 & 0\\
   \bottomrule
       \end{tabular}
     \end{center}
   \end{table}

\begin{lemma}\label{lem_2}
If $T\in \Gamma(n,d)$, then $m_{T}[0,1)=(d+1)/3$.
\end{lemma}
\begin{proof}
   According to the definition of $\Gamma(n,d)$, we may assume that $d=3k-1$ and 
   $$T\cong H_{d}(n_{1},n_{2},\ldots,n_{k}),$$ 
   where $k\geq 1$ and $n_{i}\geq 0$ for $1\leq i\leq k$.
   We prove this lemma by induction on $n\geq d+1$. If $n=d+1$, then $T\cong P_{d+1}$, and hence $m_{T}[0,1)=(d+1)/3$.

   Suppose now that $n\geq d+2$. Then $\sum_{i=1}^{k}n_{i}=n-(d+1)\geq 1$. This implies that $n_{s}\geq 1$ for some $1\leq s\leq k$. 
   Let us consider the tree $$T'\cong H_{d}(n'_{1},n'_{2},\ldots,n'_{k})$$
   where $n'_{i}=n_{i}-1$ if $i=s$, and $n'_{i}=n_{i}$ otherwise. Then $T'$ is obtained from $T$ by deleting a pendant vertex.
   Clearly, $T'\in \Gamma(n-1,d)$. By the induction hypothesis, 
   we have $m_{T'}[0,1)=(d+1)/3=k$. This implies that $\mu_{k}(T')<1$ and $\mu_{k+1}(T')\geq 1$.
   The property (P5) shows that $T'$ contains 1 as its Laplacian eigenvalue with multiplicitity $n-1-d\geq 1$. Set $q=n-d-1$.
   It follows that
   \begin{eqnarray}\label{eq6}
      \mu_{k}(T')<1, \mu_{k+1}(T')=\cdots=\mu_{k+q}(T')=1,\mu_{k+q+1}(T')>1.
   \end{eqnarray}
   Assume that the tree $T^{*}$ is isomorphic to the union of $T'$ and an isolated vertex.
   One can see that the Laplacian eigenvalues of $T^{*}$ consist of the Laplacian eigenvalues of $T'$ and an additional 0.
   According to (\ref{eq6}), we obtain that
   \begin{eqnarray}\label{eq7}
      \mu_{k+1}(T^{*})<1, \mu_{k+2}(T^{*})=\cdots=\mu_{k+q+1}(T^{*})=1,\mu_{k+q+2}(T^{*})>1.
   \end{eqnarray}
   Note also that the tree $T^{*}$ can be obtianed from $T$ by deleting a pendant edge. It follows from Proposition \ref{pro_1} that
   \begin{eqnarray}\label{eq8}
      \mu_{k}(T)\leq \mu_{k+1}(T^{*})\leq \mu_{k+1}(T)\leq
      \cdots\leq \mu_{k+q+1}(T^{*})\leq \mu_{k+q+1}(T)\leq \mu_{k+q+2}(T^{*})\leq \mu_{k+q+2}(T).
   \end{eqnarray}
   Combining (\ref{eq7}) and (\ref{eq8}), we obtain that 
   $$\mu_{k}(T)<1, \mu_{k+1}(T)\leq 1, \mu_{k+2}(T)=\cdots=\mu_{k+q}(T)=1, \mu_{k+q+1}(T)\geq 1, \mu_{k+q+2}(T)>1.$$
   According to the property (P5), 1 is the Laplacian eigenvalue of $T$ with multiplicitity $n-d=q+1$.
   This deduces that $\mu_{k+1}(T)=1$, and so $m_{T}[0,1)=k=(d+1)/3$.
   Thus we complete the proof.
\end{proof}

The above lemma shows that the trees in $\Gamma(n,d)$ satisfy $m_{T}[0,1)=(d+1)/3$. 
Next, we show that only such trees satisfy the equality. Before doing so, we introduce a tridiagonal matrix.
Let $M_{n}$ be a tridiagonal matrix of order $n$ as follows:
\begin{eqnarray*}
   M_{n}=\begin{bmatrix}
      1 & -1 \\
      -1 & 1 & -1 \\
      & -1 & \ddots & \ddots \\
      & & \ddots & \ddots  & -1 \\
      & & & -1 & 0
   \end{bmatrix}.
\end{eqnarray*}
The determinant of $M_{n}$ satisfies the following recurrence relation
\begin{eqnarray*}
   |M_{n}|=|M_{n-1}|-|M_{n-2}|.
\end{eqnarray*}
By solving this recurrence relation, one can obtain that
\begin{eqnarray}\label{eq10}
   |M_{n}|=\left\{
   \begin{aligned}
      0 &, &~\text{if}~n\equiv 1~(\text{mod}~3),\\
      -1 &,&~\text{if}~n\equiv 2~(\text{mod}~6)~\text{or}~n\equiv 3~(\text{mod}~6),  \\
      1 &, & ~\text{if}~n\equiv 0~(\text{mod}~6)~\text{or}~n\equiv 5~(\text{mod}~6).
   \end{aligned}
   \right.
\end{eqnarray}

\begin{theorem}\label{the_5}
   Let $T$ be a tree on $n$ vertices with diameter $d$. Then $m_{T}[0,1)=(d+1)/3$ if and only if $T\in \Gamma(n,d)$.
\end{theorem}
\begin{proof}
   Clearly, it follows from Theorem \ref{the_2} that $m_{T}[0,1)\geq (d+1)/3$.
   By Lemma \ref{lem_2}, it suffices to show that $m_{T}[0,1)>(d+1)/3$ if $T\notin \Gamma(n,d)$.
   If $d\not\equiv 2~(\text{mod}~3)$, then $\mu_{T}[0,1)>(d+1)/3$ as $(d+1)/3$ is not an integer.
   Suppose now that $d\equiv 2~(\text{mod}~3)$. Let $d=3k-1$ with $k\geq 1$.
   Assume that $P=v_{1}v_{2}\cdots v_{3k}$ be a diameter-path in $T$.
   Let $B_{1}=\{v_{i}\in V(P):i\equiv 2~(\text{mod}~3)\}$ and $B_{2}=V(P)\backslash B_{1}$.
   Denote by $A$ the set of vertices not in the path $P$, that is, $A=V(T)\backslash V(P)$.
   Let $$A_{1}=\{u\in A: u\sim v_{i}~\text{for some}~v_{i}\in B_{1}\}$$
   and $$A_{2}=\{u\in A: u\sim v_{i}~\text{for some}~v_{i}\in B_{2}\}.$$
   Set $A_{3}=A\backslash(A_{1}\cup A_{2})$. Note that $A_{1}$, $A_{2}$ and $A_{3}$ are disjoint sets, and $A=A_{1}\cup A_{2}\cup A_{3}$.
   If $A_{2}=\emptyset$ and $A_{3}=\emptyset$, then clearly $T\in \Gamma(n,d)$.
   So we may divide the proof into the following cases.

   \medskip

   \noindent{\bf Case 1}: $A_{2}\neq \emptyset$.

   \medskip

   Choose a vertex $u$ in $A_{2}$. Thus the vertex $u$ is adjacent to a vertex $v_{s}\in V(P)$, where $3\leq s\leq 3k-2$ and $s\not\equiv 2~(\text{mod}~3)$.
   Let $T'$ be the subgraph of $T$ induced by $V(P)\cup \{u\}$. Obviously, $uv_{s}$ is a pendant edge in $T'$.
   Let $H\cong T'-uv_{s}$. One can see that $H\cong P_{3k}\cup K_{1}$. Recall that $\mu_{k+1}(P_{3k})=1$.
   Since the Laplacian eigenvalues of $H$ consist of the Laplacian eigenvalues of $P_{3k}$ and an additional 0,
   it follows that $\mu_{k+2}(H)=\mu_{k+1}(P_{3k})=1$.
   By Proposition \ref{pro_1}, we obtain that $\mu_{k+1}(T')\leq \mu_{k+2}(H)$, and so $\mu_{k+1}(T')\leq 1$.
   We will show that 1 can not be a Laplacian eigenvalue of $T'$. If not, suppose that $1$ is a Laplacian eigenvalue of $T'$ with an eigenvector $\mathtt{x}$.
   Thus we have $L(T')\mathtt{x}=\mathtt{x}$, which implies that
   \begin{eqnarray}\label{eq9}
      (L(T')-I)\mathtt{x}=\mathbf{0}.
   \end{eqnarray}
   One can see that $L(T')-I$ can be written as:
   \begin{eqnarray}
      L(T')-I=\left[
\begin{array}{c|c|c|c} 
    M_{s-1}  &  \mathbf{0} & \begin{matrix}-1\\\mathbf{0}\end{matrix}   & \mathbf{0}\\
\hline
\mathbf{0}   & M_{3k-s} &\begin{matrix}-1\\\mathbf{0}\end{matrix}   &  \mathbf{0} \\
        \hline
        \begin{matrix}-1&\mathbf{0}\end{matrix}&   \begin{matrix}-1&\mathbf{0}\end{matrix} &   2     &  -1  \\
        \hline
        \mathbf{0}    &    \mathbf{0}  & -1  &  0  
\end{array}
\right],
\end{eqnarray}
where $M_{s-1}$ and $M_{3k-s}$ are tridiagonal matrices defined above, and the boldface $\mathbf{0}$ means a zero matrix of appropriate order.
By the elementary transformation of the matrix, we have
\begin{eqnarray}
   |L(T')-I|=\left|
\begin{array}{c|c|c|c} 
 M_{s-1}  &  \mathbf{0} & \begin{matrix}-1\\\mathbf{0}\end{matrix}   & \mathbf{0}\\
\hline
\mathbf{0}   & M_{3k-s} &\begin{matrix}-1\\\mathbf{0}\end{matrix}   &  \mathbf{0} \\
     \hline
     \mathbf{0}&   \mathbf{0} &   2     &  -1  \\
     \hline
     \mathbf{0}    &    \mathbf{0}  & 0  &  -\frac{1}{2}  
\end{array}
\right|=-|M_{s-1}|\cdot|M_{3k-s}|.
\end{eqnarray}
Since $s\not\equiv 2~(\text{mod}~3)$, we have $s-1\not\equiv 1~(\text{mod}~3)$ and $3k-s\not\equiv 1~(\text{mod}~3)$. It follows from (\ref{eq10}) that 
$$|L(T')-I|=-|M_{s-1}|\cdot|M_{3k-s}|=1\neq 0.$$
Combining (\ref{eq9}), it follows that $\mathtt{x}$ is a zero vector, a contradiction.
Hence 1 can not be a Laplacian eigenvalue of $T'$. Therefore, $\mu_{k+1}(T')<1$, that is, $m_{T'}[0,1)\geq k+1$.

It is easy to see that the tree $T'$ can be obtained from $T$ by deleting successively $n-(3k+1)$ pendant vertices.
By applying repeatedly Lemma \ref{lem_1}, we obtain that $m_{T}[0,1)\geq k+1=1+(d+1)/3>(d+1)/3$, as required.

\medskip

\noindent{\bf Case 2}: $A_{2}=\emptyset$ and $A_{3}\neq \emptyset$.

\medskip

In this case, $A_{3}$ must have a vertex which is adjacent to some vertex in $A_{1}$.
Let $w\in A_{3}$ and $u\in A_{1}$ be two vertices such that $w\sim u$ and $u\sim v_{s}$, where $5\leq s\leq 3k-4$ and $s\equiv 2~(\text{mod}~3)$.
Consider the subtree $T'$ of $T$ induced by $V(P)\cup\{w,u\}$. It is easy to see that $T'-v_{s-1}v_{s}\cong P_{s-1}\cup P_{3k-s+3}$.
By Theorem \ref{the_2}, we have 
$$m_{P_{s-1}}[0,1)\geq \frac{s-1}{3}~~~\text{and}~~~m_{P_{3k-s+3}}[0,1)\geq \frac{3k-s+3}{3}.$$
Since $s\equiv 2~(\text{mod}~3)$, we may assume that $s=3t+2$ for $1\leq t\leq k-2$. It follows that 
$$m_{P_{s-1}}[0,1)\geq t+1~~~\text{and}~~~m_{P_{3k-s+3}}[0,1)\geq k-t+1,$$
and hence $m_{P_{s-1}}[0,1)+m_{P_{3k-s+3}}[0,1)\geq k+2$. This implies that $T'-v_{s-1}v_{s}$ contains at least $k+2$ Laplacian eigenvalues less than 1.
Therefore,
$$\mu_{k+2}(T'-v_{s-1}v_{s})<1.$$
According to Proposition \ref{pro_1}, we obtain that 
$$\mu_{k+1}(T')\leq \mu_{k+2}(T'-v_{s-1}v_{s})<1,$$
which implies that $m_{T'}[0,1)\geq k+1$. 
Note that the tree $T'$ can be obtained from $T$ by deleting successively $n-(3k+2)$ pendant vertices.
By repeated application of Lemma \ref{lem_1}, it follows that $m_{T}[0,1)\geq k+1=1+(d+1)/3>(d+1)/3$, which completes the proof.
\end{proof}

Finally, we give the proof of Theorem \ref{the_4}.

\begin{proof}[Proof of Theorem \ref{the_4}]
If $m_{T}[0,1)=(d+1)/3$, then it follows from Theorem \ref{the_5} that $T\in \Gamma(n,d)$. According to the property (P2),
$\gamma=(d+1)/3$.

Conversely, suppose that the domination number of $T$ is equal to $(d+1)/3$. By Theorem \ref{the_1}, we have $m_{T}[0,1)\leq (d+1)/3$.
On the other hand, Theorem \ref{the_2} implies that $m_{T}[0,1)\geq (d+1)/3$, and hence $m_{T}[0,1)=(d+1)/3$.
\end{proof}

\begin{figure}[tbp]
  \centering
  \includegraphics[scale=0.9]{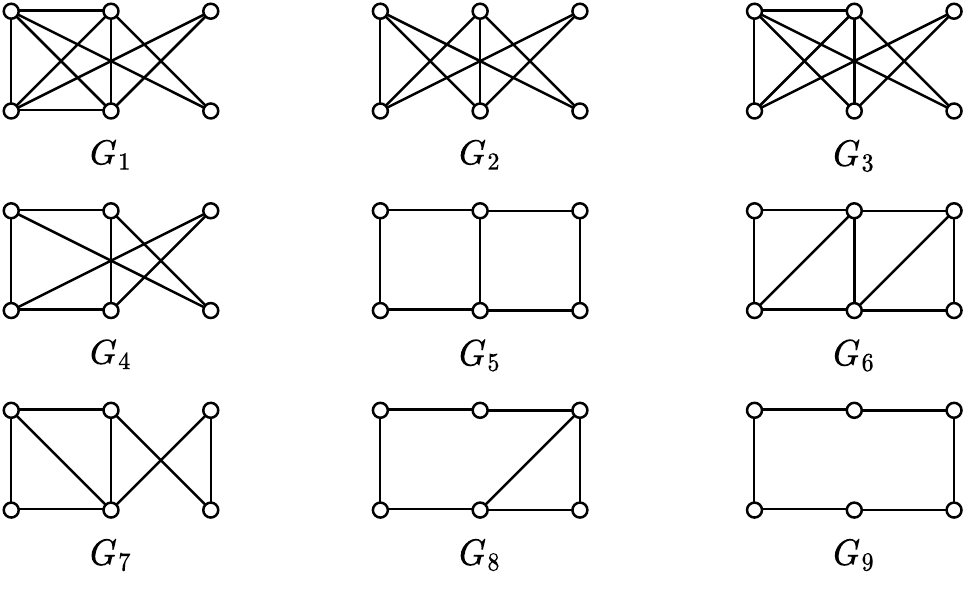}
  \vspace{-0.3cm}
  \caption{Minimum counterexamples.}
  \label{fig_2}
  \end{figure}

As mentioned above, any tree $T$ in $\Gamma(n,d)$ has domination number $\gamma=(d+1)/3$. 
This implies that these trees attain the lower bound for domination number in Theorem \ref{the_3}.
Moreover, the domination number is equal to $m_{T}[0,1)$, then $\Gamma(n,d)$ is also a set of trees which satisfy the equality in Theorem \ref{the_1}.

\section{Observation and discussion}

Given a tree $T$, Theorem \ref{the_2} shows that $(d+1)/3$ is the lower bound for $m_{T}[0,1)$, which also established 
the relation between the Laplacian eigenvalue distribution and the diameter. 
Since $m_{T}[0,1)$ means the number of its Laplacian eigenvalues less than 1, it must be an integer.
Note also that $(d+1)/3$ is an integer only if $d\equiv 2~(\text{mod}~3)$. Then Theorem \ref{the_2} can be improved slightly as follows.

   \begin{table}[tbp]
        \begin{center} \caption{Laplacian eigenvalues of $G_{1}-G_{9}$.}\label{tab_2}
          \begin{tabular}{ccccccc}
          \toprule
          & $\mu_{1}$ & $\mu_{2}$ & $\mu_{3}$ & $\mu_{4}$ & $\mu_{5}$ & $\mu_{6}$ \\
          \midrule
          $G_{1}$& 5.562 & 5.000 & 5.000 & 3.000 & 1.438 & 0\\
          $G_{2}$& 5.562 & 3.000 & 3.000 & 3.000 & 1.438 & 0\\
          $G_{3}$& 5.562 & 5.000 & 3.000 & 3.000 & 1.438 & 0\\
          $G_{4}$& 5.000 & 4.000 & 3.000 & 3.000 & 1.000 & 0\\
          $G_{5}$& 5.000 & 3.000 & 3.000 & 2.000 & 1.000 & 0\\
          $G_{6}$& 5.543 & 5.000 & 3.471 & 3.000 & 1.186 & 0\\
          $G_{7}$& 5.278 & 4.317 & 3.000 & 2.295 & 1.109 & 0\\
          $G_{8}$& 4.414 & 4.000 & 3.000 & 1.586 & 1.000 & 0\\
          $G_{9}$& 4.000 & 3.000 & 3.000 & 1.000 & 1.000 & 0\\
      \bottomrule
          \end{tabular}
        \end{center}
      \end{table}

\begin{theorem}\label{the_6}
   If $T$ is a tree with diameter $d$, then $m_{T}[0,1)\geq \lceil (d+1)/3\rceil$.
\end{theorem}

Obviously, it is equivalent to saying that a tree contains at least $\lceil (d+1)/3\rceil$ Laplacian eigenvalues less than one.
Naturally, one may ask whether a general connected graph has at least $\lceil (d+1)/3\rceil$ Laplacian eigenvalues less than one.
Let $G$ be a connected graph on $n$ vertices with diameter $d$. By computing directly, we obtain that 
the graph $G$ also satisfies the relation $m_{G}[0,1)\geq \lceil (d+1)/3\rceil$ if $G$ has at most $5$ vertices. 
However, in general, it is false for connected graphs with $n\geq 6$ vertices. 
By searching for small graphs, the minimum counterexample has 6 vertices.
Figure \ref{fig_2} exhibits all connected graph on 6 vertices which do not satisfy the inequality $m_{G}[0,1)\geq \lceil (d+1)/3\rceil$.
The Laplacian eigenvalues of these graphs are presented in Table \ref{tab_2}. One can see that, for $1\leq i\leq 9$, $G_{i}$ contains only one Laplacian eigenvalue less than 1, 
that is, $m_{G_{i}}[0,1)=1$. However, the diameter of $G_{i}$ is $d=3$, which implies that $m_{G_{i}}[0,1)<\lceil (d+1)/3\rceil$.

We remark that the lower bound in Theorem \ref{the_6} is sharp. If $d\equiv 2~(\text{mod}~3)$, Theorem \ref{the_5} presents all trees which attain the bound with equality.
A tree is called \emph{double starlike} if it has exactly two vertices of degree greater than two.
Let $T(d,p,q)$ be a tree obtained from $P_{d-1}$ by adding $p$ and $q$ pendant vertices on the end vertices of $P_{d-1}$ respectively, as shown in Figure \ref{fig_3}.
Obviously, $T(d,p,q)$ is double starlike with diameter $d$. If $p=q$, Liu, Zhang and Lu \cite{Liu2009} proved that the double starlike tree $T(d,p,q)$ is determined by its
Laplacian spectrum. In the following, we will show how many Laplacian eigenvalues of the double starlike tree $T(d,p,q)$ are less than 1.

\begin{proposition}
   The double starlike tree $T(d,p,q)$ has exactly $\lceil (d+1)/3\rceil$ Laplacian eigenvalues less than 1.
\end{proposition}
\begin{proof}
   Let $T\cong T(d,p,q)$. Since the diameter of $T$ is $d$, Theorem \ref{the_6} implies that $m_{T}[0,1)\geq \lceil (d+1)/3\rceil$.
   It suffices to show that $m_{T}[0,1)\leq \lceil (d+1)/3\rceil$. This can be proved by induction on $d$.
   If $d=2$, then $T$ is isomorphic to the star $K_{1,p+q}$, and the result follows from Lemma \ref{lem_2} directly.
   Let $T^{*}\cong T-v_{1}v_{2}$. Clearly, there are two components in $T^{*}$. 
   If $d=3$ or 4, then $T^{*}\cong K_{1,p}\cup K_{1,s}$, where $s\in \{q,q+1\}$. 
   Lemma \ref{lem_2} shows that $m_{K_{1,p}}[0,1)=m_{K_{1,s}}[0,1)=1$. 
   It follows from Proposition \ref{pro_1} that 
   $$m_{T}[0,1)\leq m_{T^{*}}[0,1)=m_{K_{1,p}}[0,1)+m_{K_{1,s}}[0,1)=2=\left\lceil\frac{ d+1}{3}\right\rceil.$$

   Assume now that $d\geq 5$. Let $T'\cong T-v_{2}v_{3}$. Suppose that $T'\cong T_{1}\cup T_{2}$, where $T_{1}\cong K_{1,p+1}$
   and $T_{2}\cong T(d-3,1,q)$. According to Lemma \ref{lem_2}, it follows that $m_{T_{1}}[0,1)=1$. 
   Note that the diameter of $T_{2}$ is equal to $d-3$. By the induction hypothesis, we have $m_{T_{2}}[0,1)\leq \lceil (d-2)/3\rceil$. 
   By Proposition \ref{pro_1}, we obtain that 
   $$m_{T}[0,1)\leq m_{T'}[0,1)=m_{T_{1}}[0,1)+m_{T_{2}}[0,1)\leq 1+\left\lceil \frac{d-2}{3}\right\rceil \leq \left\lceil\frac{ d+1}{3}\right\rceil,$$
   as required.
\end{proof}

\begin{figure}[tbp]
  \centering
  \includegraphics[scale=0.9]{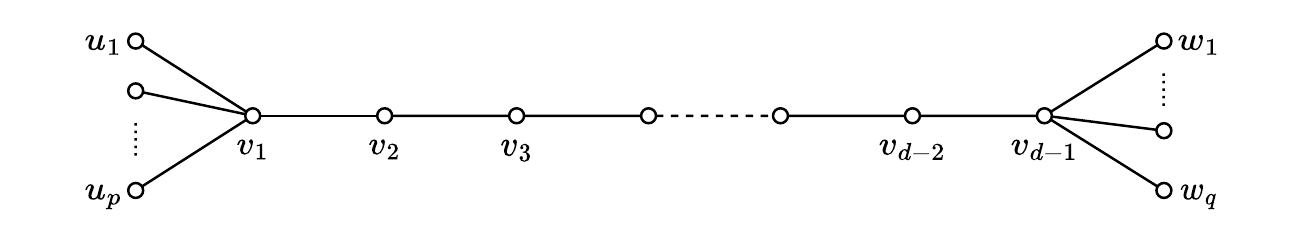}
  \caption{Double starlike tree $T(d,p,q)$.}
  \label{fig_3}
\end{figure}

Although we present an infinite family of trees which achieve the lower bound in Theorem \ref{the_6}, 
there are only a few trees with the property among all trees with given order. Let $\mathcal{T}_{n}$ be the set of all trees of order $n$.
Denote by $$\mathcal{T}_{n}^{*}=\{T\in \mathcal{T}_{n}: m_{T}[0,1)=\lceil (d+1)/3\rceil\}.$$
In other words, $\mathcal{T}_{n}^{*}$ means the set of all $n$-trees achieving the lower bound in Theorem \ref{the_6}. 
Using SageMath \cite{SageMath}, we check the Laplacian eigenvalue distribution of all trees with at most 20 vertices, see Table \ref{tab_3}.
Based on our observation, we propose the following conjecture for the ratio $\#\mathcal{T}_{n}^{*}/\#\mathcal{T}_{n}$.

\begin{conjecture}
   Let $\mathcal{T}_{n}$ and $\mathcal{T}_{n}^{*}$ be the set defined above. Then
   $$\lim_{n\to \infty}\frac{\#\mathcal{T}_{n}^{*}}{\#\mathcal{T}_{n}}=0.$$
\end{conjecture}

Indeed, if the above conjecture is true, the lower bound in Theorem \ref{the_6} can be improved as follows.

\begin{conjecture}
   Almost all trees have at least $\lceil (d+1)/3\rceil+1$ Laplacian eigenvalues less than 1.
\end{conjecture}

\begin{table}[tbp]  
    \begin{center}\caption{The ratio $\#\mathcal{T}_{n}^{*}/\#\mathcal{T}_{n}$ for $5\leq n\leq 20$.}\label{tab_3}
      \begin{tabular}{crrc}
      \toprule
      $n$ & $\#\mathcal{T}_{n}$ & $\#\mathcal{T}_{n}^{*}$ & $\#\mathcal{T}_{n}^{*}/\#\mathcal{T}_{n}$ \\
      \midrule
      5 & 3 & 2 & 0.666666667\\
      6 & 6 & 5 & 0.833333333\\
      7 & 11 & 7 & 0.636363636\\
      8 & 23 & 12 & 0.521739130\\
      9 & 47 & 20 & 0.425531915\\
      10 & 106 & 33 & 0.311320755\\
      11 & 235 & 52 & 0.221276596\\
      12 & 551 & 86 & 0.156079855\\
      13 & 1301 & 137 & 0.105303613\\
      14 & 3159 & 222 & 0.070275404\\
      15 & 7741 & 353 & 0.045601343\\
      16 & 19320 & 568 & 0.029399586\\
      17 & 48629 & 900 & 0.018507475\\
      18 & 123867 & 1433 & 0.011568860\\
      19 & 317955 & 2260 & 0.007107924\\
      20 & 823065 & 3574 & 0.004342306\\
  \bottomrule
      \end{tabular}
    \end{center}
  \end{table}

In a general way, this conjecture implies that a tree with large diameter always has many small Laplacian eigenvalues.
The distribution of small Laplacian eigenvalues was studied in \cite{Braga2013,Guo2007,Zhou2015}. 
In 2011, Trevisan, Carvalho, Vecchio and Vinagre \cite{Trevisan2011} conjectured that a tree with $n$ vertices has at least $\lceil n/2\rceil$
Laplacian eigenvalues less than $2-2/n$. Very recently, this conjecture was confirmed independently by Sin \cite{Sin2020} and Jacobs, Oliveira and Trevisan \cite{Jacobs2021}.

In summary, we obtain that for a tree, $\lceil (d+1)/3\rceil$ is the lower bound on the number of its Laplacian eigenvalues in interval $[0,1)$.
Moreover, we provide a double starlike tree attaining the lower bound. However, as mentioned at the beginning of the section, this lower bound 
is not applicable to general connected graphs. Inspired by these observations, we propose the following two problems:

\begin{itemize}
   \item For any connected graph $G$, find the suitable lower bound for $m_{G}[0,1)$.
   \item Characterize all trees which have exactly $\lceil (d+1)/3\rceil$ Laplacian eigenvalues less than 1.
\end{itemize}



\end{document}